\documentclass[12pt]{amsart}
\usepackage{amssymb}
\usepackage{amsmath}
\usepackage{amsthm}

\DeclareMathOperator{\ad}{ad}

\newcommand{\E}{{\mathcal E}}

\newcommand{\p}{\partial}

\newcommand{\Z}{\ensuremath{\mathbb{Z}}}
\newtheorem{lemma}{Lemma}

\newtheorem{theorem}{Theorem}

\begin{document}

\title[An explicit formula for a star product]
{An explicit formula for a star product with separation of variables}
\dedicatory{Dedicated to the memory of Nikolai Neumaier}
\author[Alexander Karabegov]{Alexander Karabegov}
\address[Alexander Karabegov]{Department of Mathematics, Abilene Christian University, ACU Box 28012, Abilene, TX 79699-8012}
\email{axk02d@acu.edu}

\begin{abstract} For a star product with separation of variables $\ast$ on a pseudo-K{\"a}hler manifold we give a simple closed formula of the total symbol of the left star multiplication operator $L_f$ by a given function $f$.
The formula for the star product $f\ast g$ can be immediately recovered from the total symbol of $L_f$.
\end{abstract}

\date{June 18, 2011}
\maketitle

\section{Introduction}
Given a vector space $W$ and a formal parameter $\nu$, we denote by $W[[\nu]]$ the space of formal vectors $w = w_0 + \nu w_1 + \nu^2 w_2 + \ldots, w_r \in W.$ One can also consider formal vectors that are formal Laurent series in $\nu$ with a finite polar part,
\[
    w = \sum_{r \geq k} \nu^r w_r
\]
with $k \in \Z$.

Let $M$ be a Poisson manifold endowed with a Poisson bracket $\{\cdot,\cdot\}$. A star product $\ast$ on $M$ is an associative product on the space $C^\infty(M)[[\nu]]$ of formal functions on $M$ given by a $\nu$-adically convergent series
\[
    f\ast g = \sum_{r=0}^\infty \nu^r C_r(f,g),
\]
where $C_r$ are bidifferential operators, $C_0(f,g) = fg$, and $C_1(f,g)-C_1(g,f)=i\{f,g\}$ (see \cite{BFFLS}). We also assume that the unit constant is the unity of the star-product $\ast$. A star product can be restricted to an open subset of $M$ and recovered from its restrictions to subsets forming an open covering of $M$. Given functions $f,g \in C^\infty(M)[[\nu]]$, denote by $L_f$ and $R_g$ the left star multiplication operator by $f$ and the right star multiplication by $g$, respectively. Then $L_f g = f \ast g = R_g f$ and the associativity of $\ast$ is equivalent to the property that $[L_f, R_g]=0$ for any $f,g$. The operators $L_f$ and $R_g$ are formal differential operators on $M$. It was proved by Kontsevich in \cite{K} that deformation quantizations exist on arbitrary Poisson manifolds.

A star product is called natural if, for each $r$, the bidifferential operator $C_r$ is of order not greater than $r$ in each of its arguments (see \cite{GR}). We call a formal differential operator $A = A_0 + \nu A_1 + \nu^2 A_2 + \ldots$ natural if the order of $A_r$ is not greater than $r$. If a star product is natural, the operators $L_f$ and $R_f$ for any $f \in C^\infty(M)[[\nu]]$ are natural. The star products of Fedosov \cite{F} and Kontsevich \cite{K} are natural.

Now let $M$ be a pseudo-K{\"a}hler manifold of complex dimension $m$ endowed with a pseudo-K{\"a}hler form $\omega_{-1}$ and the corresponding Poisson bracket $\{\cdot,\cdot\}$. A star product with separation of variables $\ast$ on $M$ is a star product such that the bidifferential operators $C_r$ differentiate the first argument in antiholomorphic directions and the second argument in holomorphic ones (see \cite{CMP1}, \cite{BW}). Star products with separation of variables appear naturally in the context of Berezin quantization (see \cite{Ber}). It was proved in \cite{BW} and \cite{CMP3} that the star products with separation of variables are natural in the sense of \cite{GR}.

A star product on a pseudo-K\"ahler manifold $M$ is a star product with separation of variables if and only if for any local holomorphic function $a$ and a local antiholomorphic function $b$ on $M$ the operators $L_a$ and $R_b$ are pointwise multiplication operators by the functions $a$ and $b$, respectively,
\[
    L_a = a, \ R_b = b.
\]
Otherwise speaking, if $f$ is a local holomorphic or $g$ is a local antiholomorphic function, then $f \ast g = fg$.

A formal form $\omega = \frac{1}{\nu}\omega_{-1} + \omega_0 + \nu \omega_1 + \ldots$ such that the forms $\omega_r, r \geq 1$, are of type (1,1) with respect to the complex structure on $M$ and may be degenerate is called a formal deformation of the pseudo-K\"ahler form $\omega_{-1}$. It was proved in \cite{CMP1} that the star products with separation of variables on a pseudo-K\"ahler manifold $(M, \omega_{-1})$ are bijectively parametrized by the formal deformations of the form $\omega_{-1}$ (see also \cite{N}).

A star product with separation of variables $\ast$ on $(M, \omega_{-1})$ corresponds to a formal deformation $\omega$ of the form $\omega_{-1}$ if for any contractible holomorphic chart $(U, \{z^k, \bar z^l\})$, where $1 \leq k,l \leq m$, and a formal potential $\Phi = \frac{1}{\nu}\Phi_{-1} + \Phi_0 + \nu \Phi_1 + \ldots$ of $\omega$ (i.e., $\omega = i\p\bar \p \Phi$) one has
\[
     R_{\nu\frac{\p\Phi}{\p \bar z^l}} = \nu\left (\frac{\p\Phi}{\p \bar z^l} + \frac{\p}{\p \bar z^l}\right).
\]
The star product with separation of variables $\ast$ parametrized by a given deformation $\omega$ of $\omega_{-1}$ can be constructed as follows. As shown in \cite{CMP1}, for any formal function $f$ on $U$ one can find a unique formal differential operator $A$ on $U$ commuting with the operators $R_{\bar z^l} = \bar z^l$ and $R_{\nu\frac{\p\Phi}{\p \bar z^l}}$ and such that $A1 = f$. This is the left multiplication operator by $f$ with respect to $\ast$, $A = L_f$. In particular, one can immediately check that
\[
   L_{\nu\frac{\p\Phi}{\p z^k}} = \nu\left (\frac{\p\Phi}{\p z^k} + \frac{\p}{\p z^k}\right).
\]
Now, for any formal function $g$ on $U$ we recover the product of $f$ and $g$ as $f \ast g = L_f g$. The local star products parametrized by $\omega$ agree on the intersections of coordinate charts and define a global star product on $M$.

We call the star product with separation of variables parametrized by the trivial deformation $\omega = \frac{1}{\nu}\omega_{-1}$ of $\omega_{-1}$ {\bf standard}.

Explicit formulas for star products with separation of variables on pseudo-K\"ahler manifolds can be given in terms of graphs encoding the bidifferential operators $C_r$ (see \cite{RT}, \cite{G}, \cite {HX}). 

In this paper we give a closed formula expressing the total symbol of the left star multiplication operator $L_f$ of the standard star product with separation of variables $\ast$ on a coordinate chart $U$ of a pseudo-K{\"a}hler manifold $M$ in terms of a family of differential operators on the cotangent bundle $T^\ast U$ acting on symbols of differential operators on $U$. One can immediately recover a formula for the star product $f\ast g$ on $U$ from the total symbol of the operator $L_f$.

\section{A recursive formula for the symbol of the left multiplication operator}

A differential operator $A$ on a real $n$-dimensional manifold $M$ can be written in local coordinates $\{x^i\}$ on a chart $U \subset M$ in a normal form,
\[
    A = p_{i_1 i_2 \ldots i_n}(x)\left(\frac{\p}{\p x^1}\right)^{i_1}\ldots \left(\frac{\p}{\p x^n}\right)^{i_n},
\] 
where summation over repeated indices is assumed. Denote by $\{\xi_i\}$ the dual fibre coordinates on $T^\ast U$. Then the total symbol of $A$ is given by the fibrewise polynomial function
\[
     \tau(A)(x,\xi) = p_{i_1 i_2 \ldots i_n}(x) \left(\xi_1\right)^{i_1} \ldots \left(\xi_n\right)^{i_n}
\]
on $T^\ast U$. The mapping $A \mapsto \tau(A)$ is a bijection of the space of differential operators on $U$ onto the space of fibrewise polynomial functions on the cotangent space $T^\ast U$. The composition of differential operators induces via this bijection an associative operation $\circ$ on the fibrewise polynomial functions on $T^\ast U$.
The composition $\circ$ of fibrewise polynomial functions $p(x,\xi)$ and $q(x,\xi)$ is given by the formula
\begin{align}\label{E:circ}
    (p \circ q)(x, \xi) = \exp \left(\frac{\p}{\p \eta_i}\frac{\p}{\p y^i}\right) p(x, \eta) q(y, \xi) \bigg |_{y=x, \eta = \xi} =\nonumber \\
  \sum_{r=0}^\infty \frac{1}{r!} \frac{\p^r p}{\p \xi_{i_1} \ldots \p \xi_{i_r}}\frac{\p^r q}{\p x^{i_1} \ldots \p x^{i_r}},
\end{align}
where the sum has a finite number of nonzero terms. If $p=p(x)$ or $q = q(\xi)$, then $p \circ q = pq$, which means that the operation $\circ$ has the separation of variables property with respect to the variables $x$ and $\xi$. Formula (\ref{E:circ}) is valid for complex coordinates as well.

Now let $\ast$ be the standard star product with separation of variables on a pseudo-K\"ahler manifold $(M,\omega_{-1})$ of complex dimension $m$. Choose a contractible coordinate chart $(U, \{z^k, \bar z^l\})$ on $M$ and let $\Phi_{-1}$ be a potential of $\omega_{-1}$ on $U$. Given a formal function $f = f_0 + \nu f_1 +\ldots$ on $U$, the left star multiplication operator $L_f$ is the formal differential operator on $U$ determined by the conditions that (i) $L_f 1 = f \ast 1 = f$, (ii) it commutes with the pointwise multiplication operators $R_{\bar z^l} = \bar z^l$, and (iii) it commutes with the operators
\[
    R_{\frac{\p \Phi_{-1}}{\p \bar z^l}} = \frac{\p \Phi_{-1}}{\p \bar z^l} + \nu \frac{\p}{\p \bar z^l}
\]
for $1 \leq l \leq m$. Also, the operator $L_f$ is natural, i.e., $L_f = A_0 + \nu A_1 + \ldots$, where $A_r$ is a differential operator on $U$ of order not greater than $r$. 

Denote by $\{\zeta_k, \bar \zeta_l\}$ the dual fibre coordinates on $T^\ast U$. We want to describe conditions (i) - (iii) on the operator $L_f$ in terms of its total symbol $F = \tau(L_f) = F_0 + \nu F_1 + \ldots$, where $F_r = \tau(A_r)$.
Condition (ii) means that $F$ does not depend on the antiholomorphic fibre variables $\bar \zeta_l$, $F = F(\nu, z, \bar z, \zeta)$.  Condition (i) means that $F|_{\zeta = 0} = f$ and $F_r|_{\zeta = 0} = f_r$. Condition (iii) is expressed as follows:
\begin{equation}\label{E:condiii}
              F\circ \left(\frac{\p \Phi_{-1}}{\p \bar z^l} + \nu \bar \zeta_l\right) = \left(\frac{\p \Phi_{-1}}{\p \bar z^l} + \nu \bar \zeta_l\right) \circ F.
\end{equation}
Using the definition (\ref{E:circ}) of the operation $\circ$ and its separation of variables property we simplify (\ref{E:condiii}):
\begin{equation}\label{E:simpiii}
    F \circ \frac{\p \Phi_{-1}}{\p \bar z^l} + \nu \zeta_l F = \frac{\p \Phi_{-1}}{\p \bar z^l} F + \nu \zeta_l F + \nu \frac{\p F}{\p \bar z^l}.
\end{equation}
We will use the conventional notation,
\[
     g_{k_1 \ldots k_r \bar l} = \frac{\p^{r+1} \Phi_{-1}}{\p z^{k_1} \ldots \p z^{k_r} \p \bar z^l}.
\]
Using (\ref{E:circ}) we simplify (\ref{E:simpiii}) further:
\begin{equation}\label{E:furtheriii}
    \sum_{r=1}^\infty \frac{1}{r!} g_{k_1 \ldots k_r \bar l} \frac{\p^r F}{\p \zeta_{k_1} \ldots \p \zeta_{k_r} }  = \nu \frac{\p F}{\p \bar z^l}.
\end{equation}
In particular, $g_{k\bar l}$ is the metric tensor corresponding to $\omega_{-1}$. We denote its inverse by $g^{\bar l k}$ and introduce the following operators:
\[
     \Gamma_r = g_{k_1 \ldots k_r \bar l}\, g^{\bar l k} \zeta_k \frac{\p^r}{\p \zeta_{k_1} \ldots \p \zeta_{k_r} }  \mbox{ and } D = \nu g^{\bar l k} \zeta_k \frac{\p}{\p \bar z^l}.
\] 
In particular, 
\[
\Gamma_1 = \zeta_k \frac{\p}{\p \zeta_k}
\]
is the Euler operator for the holomorphic fibre variables. Multiplying both sides of (\ref{E:furtheriii}) by $g^{\bar l k}\zeta_k$ and summing over the index $l$, we obtain the formula
\begin{equation}\label{E:form}
     \sum_{r=1}^\infty \frac{1}{r!} \Gamma_r F = D F. 
\end{equation}

We want to assign a grading to the variables $\nu$ and $\zeta_k$ such that $|\nu| = 1$ and $|\zeta_k| = -1$. 
Denote by $\E_p$ the space of formal series in the variables $\nu$ and $\zeta_k$ with coefficients in $C^\infty(U)$ such that the grading of each monomial $f(z,\bar z)\nu^r \zeta_{k_1}\ldots \zeta_{k_s}$ in such a series satisfies $r-s \geq p$.
The spaces $\E_p$ form a descending filtration on the space $\E := \E_0$:
\[
\E = \E_0 \supset \E_1 \supset \ldots.
\]
Since $L_f$ is a natural operator, its total symbol $F = \tau(L_f)$ is an element of $\E$. The operator $\Gamma_r$ acts on $\E$ and raises the filtration by $r-1$. The operator $D$ acts on $\E$ and respects the filtration. Observe that the series on the left-hand side of (\ref{E:form}) converges in the topology induced by the filtration on $\E$.
The space $\E$ breaks into the direct sum of subspaces, $\E = \E' \oplus \E''$, where $\E'$ consists of the elements of $\E$ that do not depend on the fibre variables $\zeta_k$, i.e., $\E' = C^\infty(U)[[\nu]]$, and $\E''$ is the kernel of the mapping $\E \ni H \mapsto H|_{\zeta=0}$. Observe that the Euler operator $\Gamma_1:\E \to \E$ respects the decomposition $\E = \E' \oplus \E''$, $\E'$ is its kernel, and $\E''$ is its image. Moreover, the operator $\Gamma_1$ is invertible on $\E''$. Every operator $\Gamma_k: \E \to \E$ maps $\E$ to $\E''$ and has $\E'$ in its kernel.

The following lemma is straightforward.
\begin{lemma}\label{L:etod}
  The operator $\exp D = \sum_{r=0}^\infty \frac{1}{r!} D^r$ acts on $\E$ and $\exp(- D)$ is its inverse operator on $\E$. The operator $\exp D$ leaves invariant the subspace $\E''$ and the operator $\exp D - 1$ maps $\E$ to $\E''$.
\end{lemma}
\begin{lemma}\label{L:gminnud}
  We have the following identity,
\[
      \Gamma_1 - D = \mathrm{e}^{D}\, \Gamma_1 \mathrm{e}^{- D}.
\]   
\end{lemma}
\begin{proof}
 The lemma follows from the fact that $[\Gamma_1, D] = D$ and the calculation
\[
   \mathrm{e}^{D}\, \Gamma_1 \mathrm{e}^{- D} = \sum_{r=0}^\infty \frac{1}{r!}\left(\ad D \right)^r \Gamma_1 = \Gamma_1 - D.
\]
\end{proof}
Using Lemma \ref{L:gminnud}, we rewrite formula (\ref{E:form}) as follows:
\begin{equation}\label{E:rewrform}
   \left(\mathrm{e}^{D}\, \Gamma_1 \mathrm{e}^{- D} +  \sum_{r=2}^\infty \frac{1}{r!} \Gamma_r\right) F = 0. 
\end{equation}
Introduce the operator
\begin{equation}\label{E:q}
   Q = -\mathrm{e}^{- D}\,\left(\sum_{r=2}^\infty \frac{1}{r!} \Gamma_r\right)\mathrm{e}^{D}
\end{equation}
on $\E$. It raises the filtration on $\E$ by one and maps $\E$ to $\E''$.
Applying the operator $\exp(- D)$ on both sides of (\ref{E:rewrform}) we obtain that
\begin{equation}\label{E:short}
    \left(\Gamma_1 - Q\right)\mathrm{e}^{- D} F = 0. 
\end{equation}
Using the decomposition $\E = \E' \oplus \E''$ and the last statement of Lemma~\ref{L:etod}, we observe that $\exp(- D)F = f + H$ for some $H \in \E''$. We can rewrite formula (\ref{E:short}) as follows:
\begin{equation}\label{E:inhomog}
    \left(\Gamma_1 - Q\right)H = Qf. 
\end{equation}
Since the operator $Q$ maps $\E$ to $\E''$ and $\Gamma_1$ is invertible on $\E''$, the operator $\Gamma_1^{-1}Q$ is well defined on $\E$ and raises the filtration by one, we obtain from (\ref{E:inhomog}) that
\begin{equation}\label{E:invert}
    \left(1 - \Gamma_1^{-1}Q\right)H = \Gamma_1^{-1}Qf. 
\end{equation}
The operator $1 - \Gamma_1^{-1}Q$ is invertible and its inverse is given by the convergent series
\[
    \left(1 - \Gamma_1^{-1}Q\right)^{-1} = \sum_{r=0}^\infty \left(\Gamma_1^{-1}Q\right)^r.
\]
We have
\begin{align*}
   F = \mathrm{e}^{D}(f + H) = \mathrm{e}^{D} \left(f + \left(\sum_{r=0}^\infty \left(\Gamma_1^{-1}Q\right)^r\right) \Gamma_1^{-1}Qf\right) = \\
\mathrm{e}^{D}\left(\sum_{r=0}^\infty \left(\Gamma_1^{-1}Q\right)^r\right)f = \mathrm{e}^{D}\left(1 - \Gamma_1^{-1}Q\right)^{-1}f.
\end{align*}

Combining these arguments we arrive at the following theorem.
\begin{theorem}\label{T:explicit}
    Given the standard star product with separation of variables on a pseudo-K\"ahler manifold $(M,\omega_{-1})$, 
a coordinate chart $U$ on $M$, and a function $f \in C^\infty(U)[[\nu]]$, then the total symbol $F = \tau(L_f)$ of the left star multiplication operator by $f$ is given by the following explicit formula,
\begin{equation}\label{E:explicit}
     F = {e}^{D}\left(1 - \Gamma_1^{-1}Q\right)^{-1}f.
\end{equation}
\end{theorem}
Now, to find the star product $f \ast g$, one has to calculate the total symbol $F$ of the operator $L_f$ using formula (\ref{E:explicit}), recover $L_f$ from $F$, and apply it to $g$, $f \ast g = L_f g$.

One can use the same formula (\ref{E:explicit}) to express the total symbol of the left multiplication operator $L_f$ of the star product with separation of variables $\ast_\omega$ corresponding to an arbitrary formal deformation $\omega$ of the pseudo-K\"ahler form $\omega_{-1}$. To this end one has to modify the operators $\Gamma_r$ and $D$ as follows. On a contractible coordinate chart $U$ find a formal potential $\Phi = \frac{1}{\nu}\Phi_{-1} + \Phi_0 + \ldots$ of the form 
$\omega$ and set
\[
     G_{k_1 \ldots k_r \bar l} : = \frac{\p^{r+1} \Phi}{\p z^{k_1} \ldots \p z^{k_r} \p \bar z^l}.
\]
Then $G_{k_1 \ldots k_r \bar l} = \frac{1}{\nu}g_{k_1 \ldots k_r \bar l} + \ldots$. Denote the inverse of $G_{k \bar l}$ by $G^{\bar l k} = \nu g^{\bar l k} + \ldots$. Now modify $\Gamma_r$ and $D$ (retaining the same notations) as follows:
\[
     \Gamma_r = G_{k_1 \ldots k_r \bar l}\, G^{\bar l k} \zeta_k \frac{\p^r}{\p \zeta_{k_1} \ldots \p \zeta_{k_r} }  \mbox{ and } D = G^{\bar l k} \zeta_k \frac{\p}{\p \bar z^l}.
\]
The Euler operator $\Gamma_1$ will not change. Define the operator $Q$ by the same formula (\ref{E:q}) with the modified $\Gamma_r$ and $D$. Observe that we get the old operators $\Gamma_r,\ D,$ and $Q$ for the trivial deformation $\omega= \frac{1}{\nu}\omega_{-1}$. One can show along the same lines that formula (\ref{E:explicit}) with the modified operators $D$ and $Q$ will be given by a convergent series in the topology induced by the filtration on $\E$ and will define the total symbol of the left star multiplication operator $L_f$ with respect to the star product~$\ast_\omega$. 

\end{document}